\documentclass[11pt,reqno]{amsart}

\usepackage{custom_preamble}

\begin{document}

\title{A random model for the Paley graph}

\author{\myname}
\address{\myaddress}
\email{\myemail}

\begin{abstract}
For a prime $p$ we define the Paley graph to be the graph with the set of vertices $\mathbb{Z}/p\mathbb{Z}$, and with edges connecting vertices whose sum is a quadratic residue. Paley graphs are notoriously difficult to study, particularly finding bounds for their clique numbers. For this reason, it is desirable to have a random model. A well known result of Graham and Ringrose shows that the clique number of the Paley graph is $\Omega(\log p\log\log\log p)$ (even $\Omega(\log p\log\log p)$, under the generalized Riemann hypothesis) for infinitely many primes $p$ -- a behaviour not detected by the random Cayley graph which is hence deficient as a random model for for the Paley graph. In this paper we give a new probabilistic model which incorporates some multiplicative structure and as a result captures the Graham-Ringrose phenomenon. We prove that if we sample such a random graph independently for every prime, then almost surely (i) for infinitely many primes $p$ the clique number is $\Omega(\log p\log\log p)$, whilst (ii) for almost all primes the clique number is $(2+o(1))\log p$. \end{abstract}

\maketitle

%\tableofcontents

\section{Introduction}\label{S:introduction}

Let $N$ be a prime and $R\subset \mathbb{Z}/N\mathbb{Z}$ be the set of quadratic residues. We define the so-called \emph{Paley sum graph} $\Gamma_R$ to be the graph with the set of vertices $\mathbb{Z}/N\mathbb{Z}$, and edges connecting vertices whose sum is in $R$. It is a well known open problem to find good asymptotics for the clique number $\omega(\Gamma_R)$ of this graph, that is the size of the largest complete subgraph.

Cohen~\cite{cohen} proved the lower bound $(1/2+o(1))\log N$ (all logarithms in this paper will be with base $2$). On the other hand, using basic Fourier analysis and standard estimates for Gauss sums, one can easily prove the upper bound $\sqrt{N}$. There are only minor improvements of this bound. For example, it can be shown \cite{sanderspaley} that for primes of the form $N = m^2+1$ for integer $m>2$, the clique number is at most $m-1$ (unfortunately, it is not even known that there are infinitely many primes of this form).

It is widely believed that the set of quadratic residues should have properties similar to a random subset of density $1/2$, which leads to an obvious question of finding the clique number of \emph{the random Cayley sum graph} $\Gamma_A$. As the notation suggests, $\Gamma_A$ is obtained by first choosing a random subset $A\subset \mathbb{Z}/N\mathbb{Z}$ by putting each element in it independently with probability $1/2$, and then joining vertices $x,y\in \mathbb{Z}/N\mathbb{Z}$ if and only if $x+y\in A$. Recently, Green and Morris \cite{green-morris} proved that with high probability $\omega(\Gamma_A)=(2+o(1))\log N$ which suggests that the clique number of the Paley sum graph might also be close to this value.

However, it is known that the clique number of the Paley sum graph is a little bit bigger than $2\log N$ for infinitely many primes $N$. Indeed, Graham and Ringrose \cite{graham-ringrose} proved that for infinitely many primes $N$ the lowest quadratic nonresidue $q$ is $\Omega(\log N\log\log\log N)$, that is at least $c\log N\log\log\log N$ for some constant $c>0$. Obviously, for these primes the set $\{1,2,\dots, q/2\}$ forms a large clique in $\Gamma_R$. Moreover, Montgomery \cite[page 122]{montgomery} proved that this result can, under generalized Riemann hypothesis, be improved to $\Omega(\log N\log\log N)$.

On the other hand, it is easily seen from the method used by Green and Morris that
$$\mathbb{P}(\omega(\Gamma_A)>10\log N)\leq 1/N^{2}.$$
It follows, by Borel-Cantelli lemma, that if we sample a random Cayley sum graph for each prime $N$, the clique number of only boundedly many of them would be greater than $10\log N$, which is in contrast to the result of Graham and Ringrose.

In this paper we introduce a different random graph model for the Paley graph. We show that this model usually gives the same clique number as in the random Cayley sum graphs, but also has the phenomena present in the result of Graham and Ringrose.

To motivate the construction, notice that the multiplicative structure present in the set of quadratic residues makes it relatively easy to have a large clique. For example, if we know that all the primes up to, say, $100$ are quadratic residues, then all the numbers $1,\dots,100$ are also quadratic residues, and hence $\{1,\dots,50\}$ spans a clique. 

This suggests the following model.

\begin{definition}
  Let $Q\geq 1$ be an integer. We define a random function $f\colon \mathbb{Z}/N\mathbb{Z}\to \{-1, 1\}$ in the following way. For every prime $p\in [1, Q]$ (here we embed $\mathbb{Z}/N\mathbb{Z}$ inside $\mathbb{Z}$ in an obvious way), we set $f(p)$ to be uniform $\pm 1$ random variable, and we make all these random variables independent. Next, we extend $f$ to be completely multiplicative on $[1, Q]$. Finally, for each $x\notin [1, Q]$ we set $f(x)$ to be again uniform $\pm 1$ random variable, independently of all other choices. We will say that $f$ is a \emph{$Q$-multiplicative random function}. Let $\Gamma_f$ be the random graph with vertex set $\mathbb{Z}/N\mathbb{Z}$ and edges connecting $x\neq y$ if and only if $f(x+y)=1$. We will say that $\Gamma_f$ is generated by $f$.
\end{definition}

Of course, the intuition here is that $f$ is a random model for the quadratic character $\left(\frac{\cdot}{N}\right)$ and thus the set $\{x\in\mathbb{Z}/n\mathbb{Z}\colon f(x)=1\}$ is a random model for the set of the quadratic residues $R$. In Section~\ref{S:independence} we will sketch why the decision to take values $f(p)$ (for small primes $p$) independently was sensible.

The majority of the paper will be devoted to proving that the clique number in our random model is with high probability of the same size as in the random Cayley graph.

\begin{theorem}\label{T:main}
    There exists a positive constant $c<1/2$, such that for $Q = c\log N \log\log N$ the following holds. Let $\epsilon>0$ and $f$ be the $Q$-multiplicative random function. Then
    $$\mathbb{P}((2-\epsilon)\log N\leq \omega(\Gamma_f)\leq (2+\epsilon)\log N)=1-o(1).$$    
\end{theorem}

We prove the upper and lower bounds from this theorem in Sections~\ref{S:upper} and \ref{S:lower}, respectively.

Let $R$ be a subset of the primes. We define its relative density in primes to be $\lim_{M\to \infty}\frac{|R\cap [1,M]|}{\pi(M)}$ (if the limit exists), where $\pi(M)$ denotes, as usual, the number of primes up to $M$. As an easy consequence of the previous theorem we will prove the following, which shows that although there will be infinitely many $N$ for which the clique number is $\Omega(\log N\log\log N)$, for most $N$ it will be about $2\log N$.

\begin{theorem}\label{T:cliqueborelcantelli}
 Let $\epsilon>0$ and $c>0$ be a constant for which Theorem~\ref{T:main} holds. For each prime $N$ sample a random graph $\Gamma_f^{(N)}$, independently for each $N$. Then, almost surely, the clique number $\omega(\Gamma_f^{(N)})$ will be at least $Q/2$ for infinitely many primes $N$. However, for $N$ lying in a set of relative density in primes equal to $1$, the clique number $\omega(\Gamma_f^{(N)})$ would be between $(2-\epsilon)\log N$ and $(2+\epsilon)\log N$.
\end{theorem}

We give the proof of this theorem in the last section.

\section{Notation}\label{S:notation}

Although most of the notation was implicitly defined in the introduction, we include it here for the reader's convenience. For sets $X,Y\subset \mathbb{Z}/N\mathbb{Z}$ we will denote their sumset by $X+Y$, so $X+Y=\{x+y\colon x\in X, y\in Y\}$. We will also consider their \emph{restriced sumset} $X\widehat{+}Y=\{x+y\colon x\in X, y\in Y, x\neq y\}$. We will use standard $O$-notation: if $f,g$ are two functions on positive integers we will write $f(n)=O(g(n))$ and $g(n)=\Omega(f(n))$ if there exists $C>0$ such that $|f(n)|\leq C|g(n)|$ for all large enough $n$. We will write $f(n)=\Theta(g(n))$ if there exist constants $c,C>0$ such that $c|g(n)|\leq |f(n)|\leq C|g(n)|$ for all large enough $n$. Additionally, $f(n)=o(g(n))$ means that $f(n)/g(n)$ tends to $0$ as $n$ tends to infinity. $[A,B]$ will, depending on the context, denote the set of all integers $n$ such that $A\leq n\leq B$, and its image in $\mathbb{Z}/N\mathbb{Z}$. Finally, for an integrable function $g\colon\mathbb{Z} \to \mathbb{R}$ and $\theta \in [0,1]$, we define the corresponding Fourier coefficient by
$$\hat{g}(\theta) = \sum_ng(n)e(-\theta n).$$
Here, as usual, $e(\psi) = e^{2\pi i \psi}$.

\section{Independence}\label{S:independence}

The purpose of this section is to somehow formalize the intuition that $(\frac{q}{p})$ are \emph{independent} for different primes $q$, which was the motivation for our model. We note that most of the results from this section are already present in the literature (see e.g.\ \cite[Proposition~9.1]{soundgranville}), although possibly in a slightly different form.

Let $x$ be an integer, and $y$ an integer to be chosen later (one should think of $y$ as being substantially smaller than $x$). For each prime $p\leq x$, let $v_{(p)}$ be a vector of all $(\frac{q}{p})$ where $q$ runs over primes less than $y$ (all instances of $p$ and $q$ in this section will denote primes). We define the counting function $N$ by setting, for each $s\in\{-1,1\}^{\pi(y)}$,
$$N(s) = \#\{p\leq x\colon v_{(p)}=s\}.$$
Notice that this can also be expressed as
\begin{equation}\label{E:prikazNs}
N(s) = 2^{-\pi(y)}\sum_{p\leq x}\prod_{q\leq y}\left(1+\textstyle{(\frac{q}{p})}s_q\right).
\end{equation}

Our aim is to prove that $\mathrm{Var}_sN(s)$ is small. Before doing that, we would like to mention that if $(v_{(p)})_{p\leq x}$ were independent random vectors, each uniformly distributed on $\{-1,1\}^{\pi(y)}$ (and hence each with independent coordinates), we would have
\begin{equation}\label{E:randomvar}
\mathbb{E}_v\mathrm{Var}_sN(s) = \Theta(2^{-\pi(y)}x/\log x). 
\end{equation}
For comparison, this is much smaller than the theoretical maximum $O(2^{-\pi(y)}x^2/(\log x)^2)$.

Let $P_y=\prod_{q\leq y}q$. Notice that we can identify the characters on $\{-1,1\}^{\pi(y)}$ with the set of divisors of $P_y$ by assigning to each divisor $m$ of $P_y$ the character $s\mapsto \prod_{q | m}s_q$.

By Parseval's formula we have
\begin{equation}\label{E:varns}
\mathrm{Var}_sN(s) = \mathbb{E}_sN(s)^2 - (\mathbb{E}_sN(s))^2 = \sum_{1\neq m | P_y}|\widehat{N}(m)|^2. 
\end{equation}
Now, for fixed $m\neq 1$ from \eqref{E:prikazNs} we have
\begin{align}
\widehat{N}(m) &= \mathbb{E}_sN(s)\prod_{q | m}s_q = 2^{-\pi(y)}\sum_{p\leq x} \prod_{q | m}\mathbb{E}_{s_q}(s_q+ \textstyle{(\frac{q}{p})}) \displaystyle{\prod_{q\nmid m}} \mathbb{E}_{s_q}(1+\textstyle{(\frac{q}{p})}s_q)\nonumber\\
&= 2^{-\pi(y)}\sum_{p\leq x}\prod_{q | m} \textstyle{(\frac{q}{p})}.\label{E:hatNm}
\end{align}

Assume for a moment that $m$ is even. Notice that by the quadratic reciprocity and the supplementary formula for $(\frac{2}{\cdot})$ we have
$$\prod_{q | m} \textstyle{(\frac{q}{p})} = (-1)^{\frac{p^2-1}{8}} \displaystyle{\prod_{2\neq q | m}} \textstyle{(\frac{p}{q})}(-1)^{\frac{p-1}2} (-1)^{\frac{q-1}2}.$$
The conclusion is that there exists a primitive Dirichlet character $\chi_m$ of modulus at most $4m$ such that the left-hand side is equal to $\chi_m(p)$ or $-\chi_m(p)$ with, of course, the same choice of $\pm$ sign for all $p$. We may assume that it is the former, since we will only be interested in bounding the absolute value of \eqref{E:hatNm}. The same conclusion follows for $m$ odd.

The following proposition is quite standard and will be our main tool in bounding $|\widehat{N}(m)|$. A proof can be found in e.g.\ \cite[Theorem 13.7]{montgomery-vaughan}.

\begin{proposition}
Let $\chi$ be a nonprincipal Dirichlet character of modulus $r$. Then, if Generalized Riemann Hypothesis holds,
$$\left|\sum_{p\leq x}\chi(p)\right| \lesssim x^{1/2}\log rx.$$
\end{proposition}

Using this and the well-known fact that $P_y = e^{O(y)}$, we get from \eqref{E:hatNm} that for $m\neq 1$
$$|\widehat{N}(m)| \lesssim 2^{-\pi(y)}x^{1/2}(\log x + y).$$
Plugging this into \eqref{E:varns} we get
$$\mathrm{Var}_sN(s) \lesssim 2^{-\pi(y)}x((\log x)^2 + y^2).$$
For $y=c\log x\log \log x$ this is comparable (within a power of $\log x$) to the situation one has in the random model described above (i.e.\ \eqref{E:randomvar}).

We can now easily recover the conditional result of Montgomery \cite{montgomery} mentioned in Section~\ref{S:introduction}.

\begin{proposition}
There exists a constant $c$ and infinitely many primes $p$ such that the smallest quadratic nonresidue modulo $p$ is at least $c\log p\log\log p$.
\end{proposition}

\begin{proof}
Let $y=c\log x \log \log x$, where $c$ is a constant to be chosen later. Obviously, it is enough to prove that for every large enough $x$ there exists a prime $p$ with desired properties between $x$ and $2x$. Moreover, by multiplicativity it is enough to prove that $(\frac{q}{p})=1$ for every prime $q \leq y$.

Using the same procedure as above, we can bound $\mathrm{Var}_sN'(s)$ where 
$$N'(s) = \#\{x < p\leq 2x\colon v_{(p)}=s\}.$$
We get
$$\mathrm{Var}_s N'(s)\lesssim 2^{-\pi(y)}x((\log x)^2 + y^2).$$
If, on the other hand, $N'((1,\dots,1))=0$, then
$$\mathrm{Var}_s N'(s) \gtrsim 2^{-3\pi(y)}x^2/(\log x)^2.$$
However, these two bounds are incompatible if $c$ is small enough.
\end{proof}

\section{Proof of the upper bound}\label{S:upper}

We now prove the upper bound from Theorem~\ref{T:main}. Throughout this section we will work with a fixed value of $\epsilon >0$, and consider the cardinality $k=(2+\epsilon)\log N$. Without loss of generality we may assume that $\epsilon$ is sufficiently small when needed.

We prove that with high probability there is no subset of $\mathbb{Z}/N\mathbb{Z}$ with $k$ elements which spans a clique in $\Gamma_f$. Obviously, this is equivalent to proving that
$$\mathbb{P}\left(\bigcup_{A\colon |A|=k}\{f(A\widehat{+}A)=1\}\right) = o(1).$$
It is evident from the definition of the function $f$, that our argument would split into two parts -- one dealing with $[1,Q]$ which is number-theoretical and the other dealing with its complement which is more additive-combinatorial. Roughly speaking, for each set $A$ of size $k$ we will decide, using the following proposition, into which of these two cases it falls.

\begin{proposition}[Trichotomy]\label{P:trichotomy}
	Let $L\geq 1$, $Q=L\log N$, $\delta>0$ and $A$ be a subset of $\mathbb{Z}/N\mathbb{Z}$ of size $k$.  Then $A$ must satisfy at least one the following possibilities:
	\begin{enumerate}
	\item (type 1 set) $|(A\widehat{+}A)\cap [1, Q]|\leq \delta|A\widehat{+}A|$.
	\item (type 2 set) There exists $A'\subset A$ such that $|A'|\geq (1-\delta)|A|$ and $(A'\widehat{+}A')\cap [1, Q]=\emptyset$.
	\item (type 3 set) There exist sets $C, D\subset A$, each contained in an interval of size $\log N$, such that $|C|, |D|\geq \frac{\delta^5}{6L^4}\log N$ and $C+D\subset [-Q,2Q]$.
	\end{enumerate}
\end{proposition}

In the proof of this proposition we will need the following well known result. For the proof see the book by Tao and Vu \cite{tao-vu}.

\begin{proposition}[Pl\"{u}nnecke-Ruzsa]\label{P:plunnecke-ruzsa}
	Let $A\subset \mathbb{Z}$ be a finite set such that $|A+A|\leq K|A|$ for some $K\geq 1$. Then for all nonnegative integers $m$ and $n$ we have
	$$|mA-nA|\leq K^{m+n}|A|.$$
\end{proposition}

We will also need the following result proven by Schoen \cite{schoen} which shows that we can easily shift between sumsets and restricted sumsets.

\begin{lemma}\label{L:schoen}
	Let $B$ be a subset of $\mathbb{Z}/N\mathbb{Z}$ of size $l$. Then 
	$$|B+B|=(1+o_{l\to\infty}(1))|B\widehat{+}B|.$$
\end{lemma}

\begin{proof}[Proof of Proposition \ref{P:trichotomy}]
	Suppose $A$ is not a type 1 set. Then
	$$2\delta |A\widehat{+}A|\leq 2Q \leq L|A|,$$
	and so by Lemma~\ref{L:schoen} we have $|A+A|\leq K|A|$ for $K=L/\delta$.
	
	We now split $\mathbb{Z}/N\mathbb{Z}$ into disjoint intervals, each of length $\log N$ (except possibly one). First of all, we prove that only few of these intervals are hit by $A$. Denote this number by $h$ and choose points $a_1,\dots, a_h\in A$, one from each of these $h$ intervals. Notice that we can choose at least $\frac{h}{2L+2}$ of these points (we denote them by $a'_1,\dots, a'_l$) such that the distance between any two of them is at least $Q+1$. Indeed, any translate of $[-Q, Q]$ intersects at most $2L+2$ intervals, and hence eliminates at most this many points. Obviously, translates $a'_i+(A\hat{+}A)\cap [1, Q]$ for $i=1, \dots, l$ are all disjoint and hence
	$$|A + ((A\hat{+}A)\cap [1, Q])|\geq \frac{h}{2L+2}\cdot \delta |A\hat{+}A|\geq \frac{\delta h |A|}{3L}.$$
	On the other hand, using Proposition~\ref{P:plunnecke-ruzsa} (Pl\"{u}nnecke-Ruzsa) we get
	$$|A + ((A\hat{+}A)\cap [1, Q])|\leq |A+A+A|\leq K^3|A|.$$
	Conclusion is that $h\leq 3K^3L/\delta$.
	
	Let $\beta = \frac{\delta^2}{3K^3L}$. We call an interval \emph{good} if it contains at least $\beta\log N$ elements from $A$. Let $A'$ be the intersection of $A$ with the union of all good intervals. Obviously, 
	$$|A\setminus A'|\leq h\cdot \beta\log N\leq \delta |A|$$
	and hence $A$ is a type 2 set if $A'\widehat{+}A'$ is disjoint from $[1,Q]$. If, on the other hand, $(A'\widehat{+}A')\cap [1, Q]\neq \emptyset$, then there exist two good intervals $I_1$ and $I_2$ such that	
	$$((A\cap I_1) \widehat{+} (A\cap I_2))\cap [1, Q] \neq \emptyset.$$
	Let $C'=A\cap I_1$ and $D'=A\cap I_2$; by the previous line we have $C'\widehat{+}D'\subset [-Q,2Q]$, and so the only thing left to do to prove that $A$ is a type 3 set is to replace the restricted sumset by a genuine sumset. Without loss of generality, we may assume that $|C'|\leq |D'|$. Let $C\subset C'$ be a set of $|C'|/2$ elements, and let $D\subset D'$ be a set $|D'|/2$ elements, disjoint from $C$. Then $C+D\subset C'\widehat{+}D'$ and we can conlude that $A$ is a type 3 set.
\end{proof}

Our main tool for dealing with type 1 and type 2 sets is the following proposition, the proof of which occupies the majority of \cite{green-morris}.

\begin{theorem}\label{T:Skm}
    For every $m$ define
    $$S_k^m = \{A\subset\mathbb{Z}/N\mathbb{Z}\colon |A|=k\text{ and }|A\widehat{+}A|=m\}.$$
    Then
    $$|S_k^m|\leq 2^{(1-\epsilon^3)m}.$$
\end{theorem}
We note that in the proof of this theorem different strategies are used depending on the size of $m$. \cite{green-morris} deals with the cases $m=O(k)$ and $m=\Omega(k^2)$, whereas the claim for other values of $m$ was already proven in \cite{green-clique}.
 
We now shift the attention to our method for dealing with type 3 sets. The rough idea is to first prove that with high probability all of the Fourier coefficients of (the restriction of) function $f$ will be quite small. On the other hand, we will show that the existence of sets $C$ and $D$ as in the definition of type 3 sets, which additionally satisfy $f(C+D)=1$, implies the existence of a large Fourier coefficient, so we will be able to conclude that this is quite unlikely to happen.

We will work with the function $g\colon \mathbb{Z}\to \mathbb{R}$ defined by
$$g(x) = \begin{cases}
               f(x\ \mathrm{mod}\ N)               & \text{for } -Q\leq x\leq 2Q,\\
               0               & \text{otherwise}.
           \end{cases}$$
           
The following proposition covers the first part of the strategy outlined above, namely that it is unlikely that $g$ has a large Fourier coefficient.

\begin{proposition}\label{P:boundsforbigFourier}
For any $l\leq Q$ we have 
$$\mathbb{P}\left(\sup_{0\leq \theta\leq 1}|\widehat{g}(\theta)| \geq l\right) \leq Q^{4+o(1)}/l^5.$$
\end{proposition}

\begin{proof}
	Let $\psi\in [0,1]$. We have
	\begin{align*}
		\mathbb{E}|\widehat{g}(\psi)|^4 &= \sum_{-Q\leq n_1,n_2,n_3,n_4\leq 2Q}\mathbb{E}g(n_1)g(n_2)g(n_3)g(n_4)\cdot e((n_1+n_2-n_3-n_4)\psi)\\
		&\leq \sum_{-Q\leq n_1,n_2,n_3,n_4\leq 2Q}\left|\mathbb{E}g(n_1)g(n_2)g(n_3)g(n_4)\right|
	\end{align*}
	Notice that the expectation appearing in the sum would be $0$ unless the product of those $n_i$s that take values inside $[1,Q]$ is a square, and no number outside this interval is equal to the odd number of $n_i$s. We will call such quadruples \emph{bad}. Obviously, there are $O(Q^2)$ bad quadruples for which all $n_i$ are outside $[1,Q]$. At the other extreme, consider bad quadruples for which all $n_i$ are inside $[1,Q]$. Their product is a square smaller than $Q^4$, which we can choose in $Q^2$ ways. Additionally, by the divisor bound we can choose four of its divisors (that is, $n_i$s) in at most $Q^{o(1)}$ ways, giving in total $Q^{2+o(1)}$ bad quadruples. The remaining case is when two of the $n_i$s are equal to a number outside $[1,Q]$, and the product of the remaining two is a square. In the same way as before, we can see that there are at most $Q^{2+o(1)}$ such bad quadruples. We can now conclude that 
	$$\mathbb{E}|\widehat{g}(\psi)|^4\leq Q^{2+o(1)}.$$
	By Markov's inequality this gives us
	\begin{equation}\label{E:Markovforghat}	
		\mathbb{P}\left(|\widehat{g}(\psi)| \geq l/2\right) \leq Q^{2+o(1)}/l^4.
	\end{equation}	
	Define
	$$\theta_j = jl/80Q^2,\quad\text{for }j=0,\dots,\lfloor 80Q^2/l\rfloor.$$
	For any $\theta\in [0,1]$ there is $j$ such that $|\theta - \theta_j| < l/80Q^2$. For such $j$ we have
	\begin{equation}\label{E:continuityofghat}
		|\widehat{g}(\theta_j) - \widehat{g}(\theta)|\leq \sum_{-Q\leq n\leq 2Q}|e((\theta_j-\theta)n) - 1| \leq l/2,
	\end{equation}
	where the last inequality follows from $|e(\alpha)-1|\leq 2\pi|\alpha|$ which holds for all real $\alpha$.
	
	From \eqref{E:Markovforghat}, \eqref{E:continuityofghat}, and the union bound we have
	$$\mathbb{P}\left(\sup_{0\leq \theta\leq 1}|\widehat{g}(\theta)| \geq l\right) \leq \mathbb{P}\left(\max_{j}|\widehat{g}(\theta_j)| \geq l/2\right) \leq Q^{4+o(1)}/l^5,$$
	and this is what we wanted to prove.
\end{proof}

The following proposition covers the second part of our strategy, namely the existence of a large Fourier coefficient. The proof is quite standard, but the brevity of the argument allows us to include it here for completeness.

\begin{proposition}\label{P:existsbigFourier}
  Let $C$ and $D$ be two sets of integers such that $g(C+D)=1$. Then 
    $$\sup_{0\leq \theta\leq 1}|\widehat{g}(\theta)|\geq |C|^{1/2}|D|^{1/2}.$$
\end{proposition}

\begin{proof}
	From the given condition we have
 	\begin{align*}
      |C||D| &= \sum_{-Q\leq x\leq 2Q}\sum_{c\in C}\sum_{d\in D}f(x)1_{x=c+d}\\
      &= \sum_{-Q\leq x\leq 2Q}\sum_{c\in C}\sum_{d\in D}g(x)\int e(\theta (x-c-d))\,d\theta\\  
      &= \int \overline{\widehat{g}(\theta)}\widehat{1_C}(\theta)\widehat{1_D}(\theta)\,d\theta\\
      &\leq \sup_{0\leq \theta\leq 1}|\widehat{g}(\theta)| \cdot \int |\widehat{1_C}(\theta)| |\widehat{1_D}(\theta)|\,d\theta
 	\end{align*}
 	By using Cauchy-Schwarz inequality and Parseval's identity we can bound this further by
 	$$\sup_{0\leq \theta\leq 1}|\widehat{g}(\theta)| \cdot \|\widehat{1_C}\|_{L^2([0,1])}\|\widehat{1_D}\|_{L^2([0,1])} = \sup_{0\leq \theta\leq 1}|\widehat{g}(\theta)| \cdot |C|^{1/2}|D|^{1/2},$$
	and this gives the inequality from the statement.
\end{proof}

We now have all the tools needed and are in the position to start the proof.

\begin{proof}[Proof of the upper bound in Theorem~\ref{T:main}]
Set $L=Q/\log N$, as in Proposition~\ref{P:trichotomy}, and $\delta=\epsilon^4$. As we mentioned before, our aim is to prove that with high probability $\omega(\Gamma_f)<k=(2+\epsilon)\log N$. Obviously, this is equivalent to proving that
\begin{equation}\label{E:main}
\mathbb{P}\left(\bigcup_{A\colon |A|=k}\{f(A\widehat{+}A)=1\}\right) =  o(1). 
\end{equation}
First of all, we apply Proposition~\ref{P:trichotomy} which shows that every set $A$ of interest is of at least one of the types 1, 2, and 3. Denote these families of sets by $\mathcal{F}_1$, $\mathcal{F}_2$, $\mathcal{F}_3$, respectively.

First we deal with type 1 sets which is quite easy. Indeed, using the notation introduced in this section, we have
\begin{align}
\mathbb{P}\left(\bigcup_{A\in\mathcal{F}_1}\{f(A\widehat{+}A)=1\}\right)&\leq \sum_{A\in\mathcal{F}_1}\mathbb{P}\left(f(A\widehat{+}A)=1\right)\leq \sum_{m\geq k-1}\sum_{A\in S_k^m\cap \mathcal{F}_1}\mathbb{P}\left(f(A\widehat{+}A)=1\right)\nonumber\\
&\leq \sum_{m\geq k-1}|S_k^m|2^{-(1-\delta)m} \leq \sum_{m\geq k-1}2^{-(\epsilon^3-\epsilon^4)m}=o(1).\label{E:o1F1}
\end{align}
Here we used Theorem~\ref{T:Skm} to obtain the last inequality. 

Let $A$ now be a type 2 set. After possibly discarding some extra elements, we see that there is a subset $A'\subset A$ of size $k'=(1-\delta)(2+\epsilon)\log N > (2+\epsilon^2)\log N$ such that $A'\widehat{+}A'$ is disjoint from $[1, Q]$. This implies that
\begin{align}
\mathbb{P}\left(\bigcup_{A\in\mathcal{F}_2}\{f(A\widehat{+}A)=1\}\right)&\leq \mathbb{P}\left(\bigcup_{|A'|=k'}\{f(A'\widehat{+}A')=1\}\right)\leq \sum_{|A'|=k'}\mathbb{P}\left(f(A'\widehat{+}A')=1\right)\nonumber\\
&= \sum_{m\geq k'-1}\sum_{A\in S_{k'}^m}\mathbb{P}\left(f(A\widehat{+}A)=1\right)= \sum_{m\geq k'-1}|S_{k'}^m|\cdot 2^{-m}\nonumber\\
&\leq \sum_{m\geq k-1}2^{-\epsilon^6m}=o(1).\label{E:o1F2}
\end{align}

Finally, we deal with type 3 sets. For these the situation is somewhat trickier since one can show that the expected number of type 3 sets that span a clique tends to infinity. Thus, the union bound (i.e.\ first moment method) used for type 1 and 2 sets wouldn't work.

Let $A$ be a type 3 set. By the definition there are subsets $C,D\subset A$, each of size $\frac{\delta^5}{6L^4}\log N$ such that $C+D\subset [-Q,2Q]$ and both $C$ and $D$ are contained in intervals of size $\log N$. Because of the last property, after possible translations, we get two subsets of $\mathbb{Z}$ whose sumset is contained in $[-Q,2Q]$. We will abuse the notation and denote these sets by $C$ and $D$ as well. By Proposition~\ref{P:existsbigFourier}, we get
$$\sup_{0\leq \theta\leq 1}|\widehat{g}(\theta)|\geq \textstyle{\frac{\delta^5}{6L^4}\log N}.$$
However, by Proposition~\ref{P:boundsforbigFourier}, we know that the probability of this is quite small. Indeed, taking $l=\frac{\delta^5}{6L^4}\log N$, we have
$$\mathbb{P}\left(\sup_{0\leq \theta\leq 1}|\widehat{g}(\theta)| \geq \textstyle{\frac{\delta^5}{6L^4}\log N}\right) \leq L^{25} / (\log N)^{1-o(1)}.$$

From this we can finally conclude that 
\begin{equation}\label{E:o1F3}
\mathbb{P}\left(\bigcup_{A\in\mathcal{F}_3}\{f(A\widehat{+}A)=1\}\right)\leq  L^{25} / (\log N)^{1-o(1)} = o(1),
\end{equation}
for $L=c\log\log N$. Combining \eqref{E:o1F1}, \eqref{E:o1F2}, and \eqref{E:o1F3} we get \eqref{E:main}.

\end{proof}

\section{Proof of the lower bound}\label{S:lower}

In this section we prove the lower bound from Theorem~\ref{T:main}. First of all, we would like to point out that the same bound also holds for random Cayley sum graphs. Although we are not aware of a proof of this fact in the literature, it is certainly a trivial consequence of the much more difficult result about the chromatic number by Green \cite{greenchromatic}. In this section we will give a proof for our model which works also, with few easy modifications, for the random Cayley sum graph model.  

We fix $k=(2-\epsilon)\log N$. Let $\mathcal{U}$ be the family of all sets $A\subset [N/4,N/2]$  of size $k$ such that $|A\widehat{+}A|={k\choose 2}$. Let $R$ be the number of sets $A\in \mathcal{U}$ such that $f(A\widehat{+}A)=1$. We will prove that with high probability $R\geq 1$ which, of course, proves that with high probability there is a clique of size $k$.

First we prove that almost all subsets of $[N/4,N/2]$ of size $k$ belong to $\mathcal{U}$.

\begin{lemma}
  $|\mathcal{U}|=(1+o(1)){{N/4}\choose k}$.
\end{lemma}
\begin{proof}
  Let $A\subset [N/4, N/2]$ be a random subset of size $k$. There are at most $(N/4)^3$ quadruples of different elements $x_1,x_2,x_3,x_4\in [N/4, N/2]$ such that $x_1+x_2=x_3+x_4$ and for each of these the probability that it is included in $A$ is $O(k^4/N^4)$. By the union bound, probability that $A$ contains at least one such quadruple is $O(k^4/N)=o(1)$ which is what we had to prove.
\end{proof}

By the previous lemma we have
\begin{equation}\label{E:eR}
\mathbb{E}R\geq \textstyle{\frac12}{{N/4}\choose k}2^{-{k\choose 2}} \geq 2^{k\log N - k^2/2 + o(k^2)}.
\end{equation}

We now focus on bounding the variance of $R$. For each $l\leq {k\choose 2}$, let $\mathcal{V}_l$ be the family of all pairs of sets $A,B\in \mathcal{V}_l$ such that $|(A\widehat{+}A)\cap (B\widehat{+}B)|=l$. We have
\begin{align}
\mathrm{Var}R &= \sum_{l=0}^{{k\choose 2}} \sum_{(A,B)\in\mathcal{V}_l}\mathrm{cov}(1_{f(A\widehat{+}A)=1}, 1_{f(B\widehat{+}B)=1})\nonumber\\
&\leq \sum_{l=1}^{k\choose 2} |\mathcal{V}_l|\mathbb{P}(f(A\widehat{+}A)=1, f(B\widehat{+}B)=1)\nonumber\\
&= 2^{-k(k-1)}\sum_{l=1}^{k\choose 2} |\mathcal{V}_l|2^l.\label{E:varRpoc}
\end{align}
Here the last equality holds because of our choice of family $\mathcal{U}$ which ensured that for every $A\in\mathcal{U}$ the set $A\widehat{+}A$ is disjoint from $[1, Q]$.

The next step is to bound $\mathcal{V}_l$. For each fixed $A\in \mathcal{U}$ we are going to bound the number of possible $B\in\mathcal{U}$ such that $|(A\widehat{+}A) \cap (B\widehat{+}B)|=l$.

Consider one such set. Let $s$ be the unique integer such that
$${s\choose 2} < l \leq {{s+1}\choose 2}.$$
We prove that we can order the elements of $B$ and find a subset $B'\subset B$ of $s$ elements such that for each $b'\in B'$ there is a smaller $b\in B$ such that $b+b'\in A\widehat{+}A$. We can do this in the following way.

Form a graph with vertex set $B$ and edges joining $b$ and $b'$ if $b+b'\in A\widehat{+}A$. For each of the connected components $B_1,\dots, B_r$ of this graph, choose an element $b_i\in B_i$. Now take an arbitrary order of elements of $B$ which satisfies the sole condition that if $b$ and $b'$ come from the same component $B_i$, and distance from $b$ to $b_i$ is less then distance from $b'$ to $b_i$, then $b<b'$. One can easily construct this by ordering component by component. Obviously, $b_1,\dots, b_r$ are the only elements in this order that don't satisfy the required condition. Given that our graph has $k$ vertices and at least $l$ edges, the problem boils down to finding the maximal possible number of connected components in such a graph. It is intuitively clear (and one can easily show) that a maximizer for this optimization problem is a graph with $k-s$ connected components, all but one of which are singletons, and the remaining one contains $s+1$ vertices and $l$ edges. In this graph, we can take $B'$ to be equal to the nontrivial component, apart from the smallest element in it. 

Using this observation, we make an enumerative argument as follows -- we can choose set $A$ in at most $(N/4)^k$ ways, ordering in $k!$ ways, positions in the order that would be occupied by elements from $B'$ in $k^{s}$ ways, elements on positions outside $B'$ in $(N/4)^{k-s}$ and elements from $B'$ in $(k^3)^{s}$ ways (each element from $B'$ is uniquely determined by choosing one of at most $k^2$ possible sums from $A\widehat{+}A$ and one of at most $k$ predecessors from $B$). Putting all this together we get
$$|\mathcal{V}_l| \leq (N/4)^kk!k^{s}(N/4)^{k-s}(k^3)^{s}\leq 2^{(2k-s)\log N + o(k^2) }.$$
Since $l\leq {k\choose 2}$ we have $s\leq k-1$ and so
$$l-s\log N \leq {{s+1}\choose 2} - s\log N \leq (\textstyle{\frac12} - \textstyle{\frac{1}{2-\epsilon}})k^2 + O(k),$$
and these two bounds together with \eqref{E:varRpoc} give
\begin{equation}\label{E:varR}
\mathrm{Var}R \leq 2^{2k\log N - (\frac12 + \frac{1}{2-\epsilon})k^2 + o(k^2)}. 
\end{equation}
By Chebyshev's inequality, \eqref{E:eR} and \eqref{E:varR} we now have 
\begin{equation}\label{E:secondmomentlowerbound}
  \mathbb{P}(R=0) \leq \frac{\mathrm{Var} R}{(\mathbb{E}R)^2} = o(1),
\end{equation}
which is what we wanted to prove.

\section{Proof of Theorem~\ref{T:cliqueborelcantelli}}\label{S:proofofcliqueborelcantelli}

Theorem~\ref{T:cliqueborelcantelli} is a straightforward consequence of Theorem~\ref{T:main} and the following strengthening of Borel-Cantelli lemma. For the proof of it, see e.g.\ the book by Chung \cite[page 83]{chung}

\begin{proposition}[Generalized second Borel-Cantelli lemma]\label{P:secondborelcantelli}
Let $(A_n)_{n\geq 1}$ be a sequence of independent events such that $\sum_n\mathbb{P}(A_n) = \infty$. Then $\mathbb{P}(A_n \text{ happens for infinitely many } n)=1$ and
$$\lim_{m\to\infty}\frac{\sum_{n=1}^{m}1_{A_n}}{\sum_{n=1}^{m}\mathbb{P}(A_n)} = 1\quad\text{almost surely.}$$
\end{proposition}

\begin{proof}[Proof of Theorem~\ref{T:cliqueborelcantelli}]
Notice that
 \begin{align*}
    \mathbb{P}(\omega(\Gamma_f^{(N)})\geq Q/2) &\geq \mathbb{P}(f([1, Q])=1) = \mathbb{P}\left(f(p)=1, \text{for all primes }p\in [1, Q]\right)\\
    &= 2^{-\pi(Q)} \geq 1/N,
\end{align*}
where the last inequality follows from the prime number theorem and the assumption $Q\leq \log N\log\log N$. The first part of the claim from the proposition follows by Borel-Cantelli lemma. To prove the second part, for each prime $N$ let 
$$A_N = \{\omega(\Gamma_f^{(N)}) <(2-\epsilon)\log N \text{ or } \omega(\Gamma_f^{(N)}) > (2+\epsilon)\log N)\}.$$
Since we have proven that $\mathbb{P}(A_N)=o(1)$, we obviously have that
$$\sum_{N\leq M, N \text{prime}}\mathbb{P}(A_N) = o(\pi(M)),$$
and the conclusion follows from Proposition~\ref{P:secondborelcantelli}.
\end{proof}

\textsl{Acknowledgements.} We would like to thank Ben Green for suggesting this project, Freddie Manners for helpful discussions, and Kannan Soundararajan for giving a reference for some of the results contained in Section \ref{S:independence}. We would also like to thank Sean Eberhard for the idea of ordering the elements of $B$ when bounding $\mathrm{Var} R$ in Section~\ref{S:lower}.

\bibliography{paley_model}{}
\bibliographystyle{alpha}

\end{document}